\documentclass[a4]{article}   	
\usepackage{hyperref}
\usepackage{amsmath,amssymb,amsthm,bbm}
\usepackage{tikz}
\usepackage{todonotes}
\usetikzlibrary{arrows,decorations.pathmorphing}
\tikzset{snake it/.style={decorate, decoration={snake, amplitude=.4mm,segment length=4mm}}}
\usepackage{color}
\usepackage[notref,notcite]{showkeys}
\newcommand{\Z}{\mathbb{Z}}
\newcommand{\mc}[1]{\mathcal{#1}}

\renewcommand{\P}{\mathbb{P}}

\newcommand{\R}{\mathbb{R}}
\newcommand{\N}{\mathbb{N}}

\newcommand{\sss}[1]{\scriptscriptstyle{#1}}
\newcommand{\blank}[1]{}

\newcommand{\bs}[1]{\boldsymbol{#1}}
\newcommand{\uE}{\underline{E}}
\newcommand{\oE}{\overline{E}}
\newcommand{\uF}{\underline{F}}
\newcommand{\oF}{\overline{F}}

\newcommand{\Pp}{\mathbb{P}_p}

\title{Phase transitions for degenerate random environments} 
\author{Mark Holmes \footnote{University of Melbourne} \& Thomas S. Salisbury\footnote{York University}}

\newcounter{thmcounter}
\newtheorem{theorem}[thmcounter]{Theorem}
\newcounter{excounter}
\newtheorem{example}[excounter]{Example}
\newcounter{condcounter}
\newtheorem{condition}[condcounter]{Condition}
\newcounter{opencounter}
\newtheorem{open}[opencounter]{Open problem}
\newcounter{lemcounter}
\newtheorem{lemma}[lemcounter]{Lemma}
\newcounter{corcounter}
\newtheorem{corollary}[corcounter]{Corollary}
\newcounter{defcounter}
\newtheorem{definition}[defcounter]{Definition}
\newcounter{propcounter}
\newtheorem{proposition}[propcounter]{Proposition}

\newcounter{rkcounter}
\newtheorem{remark}[rkcounter]{Remark}
\newtheorem*{theorem*}{Theorem}

\begin{document}
\maketitle

\begin{abstract}
We study a class of models of i.i.d.~random environments in general dimensions $d\ge 2$, where each site is equipped randomly with an environment, and a parameter $p$ governs the frequency of certain environments that can act as a barrier.  We show that many of these models (including some which are non-monotone in $p$) exhibit a sharp phase transition for the geometry of connected clusters as $p$ varies.
\end{abstract}

\section{Introduction} 
Fix $d\ge 2$, and set $[d]=\{1,2,\dots, d\}$.   Let $\mc{E}_+=\{e_i\}_{i \in [d]}$ 
%=\{e_1,e_2,\dots, e_d\}$ 
denote the set of canonical basis vectors for $\Z^d$ and let $\mc{E}_-=\{-e_i\}_{i \in [d]}$ and $\mc{E}=\mc{E}_+\cup \mc{E}_-$.  The \emph{orthant model} is the name we give to the random directed graph in which a vertex $x\in\mathbb{Z}^d$ either connects (with probability $p$) to all $x+e$, $e\in\mc{E}_+$, or (with probability $1-p$) to each $x+e$, $e\in\mc{E}_-$. Our motivation was to show two properties of the orthant model in all dimensions.
\begin{itemize}
\item That ``filling in'' the holes in the forward cluster of the origin $o$ yields a cluster bounded by sites of type $\mc{E}_+$ (and this filled-in region is in turn the forward cluster of a different model, which we will call the \emph{half-orthant model});
\item That this cluster undergoes a phase transition in $p$;
%\item That a shape theorem holds for this cluster, at least when $p$ is large.
\end{itemize}
Both statements will be special cases of more general results (i.e.~for a broader class of models) that we describe below.
%But since many of our arguments apply more generally, we will state and prove our results for a broader class of models. We start by defining those.

\subsection{Models and main results}

Let $\mu$ be a probability measure on the power set of $\mc{E}$.  Let $(\mc{G}_x)_{x \in \Z^d}$ be i.i.d.~with law $\mu$.  This induces a random directed graph on $\Z^d$ - insert arrows from $x$ to each of the vertices $\{x+e:e\in \mc{G}_x\}$.  We are interested in the set of vertices $\mc{C}_x\subset \Z^d$ that can be reached from $x$ by following arrows, as well as the sets $\mc{B}_x=\{y \in \Z^d:x \in \mc{C}_y\}$ and $\mc{M}_x=\mc{C}_x\cap \mc{B}_x$.  These models are examples of {\em degenerate random environments} -- see \cite{DRE, DRE2}. The study of these environments lays the foundation for understanding {\em random walks in non-elliptic random environments}.  See \cite{BSz, Hughes, Zeit04} for the uniformly elliptic theory and \cite{RWDRE,RWDRE2} (together with the references in the latter) for the non-elliptic case. In this context the arrows from $x$ represent the  possible steps that the walk can take from $x$.  Then the condition that $\mc{C}_x$ is infinite for every $x$ is precisely the condition which ensures that the random walker does not get stuck on a finite set of sites (see e.g.~\cite[Lemma 2.2]{RWDRE}).  This is the setting that interests us.  

To state our main results we introduce an explicit probability space (with a particular coupling structure) on which our models are defined.  Let $(\Omega, \mc{F}, \P)$ be a probability space on which $(U_x)_{x \in \Z^d}, (U'_x)_{x \in \Z^d}$ are i.i.d.~$U(0,1]$ random variables.  

Let $k,\ell\in \N$, $\bs{E}=(E_1,\dots, E_k)$, $\bs{F}=(F_1,\dots, F_\ell)$ with each $E_i,F_j \subset \mc{E}$.  Let $\mc{D}_k=\{(r_1, \dots, r_k): r_i\ge 0 \text{  for each $i$ and }\sum_{i=1}^k r_i=1\}$.  For $p\in [0,1]$, $\bs{r}\in \mc{D}_k$,  $\bs{q}\in \mc{D}_\ell$ and  $x \in \Z^d$, set 
\begin{equation}
\mc{G}_x=\begin{cases}
E_i, & \text{ if } U_x< p\text{ and }U'_x \in [\sum_{j=1}^{i-1} r_j,\sum_{j=1}^{i}r_j)\\
F_i, & \text{ if }U_x\ge p \text{ and }U'_x\in [\sum_{j=1}^{i-1} q_j,\sum_{j=1}^{i}q_j).
\end{cases}
\end{equation}
%This induces a random directed graph on $\Z^d$ - insert arrows from $x$ to each of the vertices $\{x+e:e\in \mc{G}_x\}$.  
We denote the $\mc{C}_x$ for this model by $\mc{C}_x(\bs{E},\bs{F},\bs{r},\bs{q},p)$. If $k=\ell=1$ then $r_1=q_1=1$ and we say that the model is \emph{2-valued} and we write $\mc{C}_x(E_1,F_1,p)$ for the forward cluster. Or, if the sets $E_1$ and $F_1$ are understood, simply $\mc{C}_x(p)$.

Let $\uE=\cap_{i=1}^k E_i$ and $\oE=\cup_{i=1}^k E_i$, and similarly $\uF=\cap_{i=1}^\ell F_i$ and $\oF=\cup_{i=1}^\ell F_i$.   Let $\Omega_+=\{x:U_x< p\}$ (these are the sites that receive an $E$ environment), and $\Omega_-=\{x:U_x\ge p\}$.   
\begin{remark}
\label{rem:subsets}
%An elementary observation is that 
If $E'_i\subset E_i$ for each $i\in [k]$ and $F'_i\subset F_i$ for each $i\in [\ell]$ then for each $\bs{r}$, $\bs{q}$, and $p$, we have $\mc{C}_o(\bs{E}',\bs{F}',\bs{r},\bs{q},p)\subset \mc{C}_o(\bs{E},\bs{F},\bs{r},\bs{q},p)$.  
\end{remark}
\begin{example}
\label{exa:orthant}
The case $k=\ell=1$, $E_1=\mc{E}_+$ and $F_1=\mc{E}_-$ is what we have referred to above as the {\em orthant model}.  The sets $\mc{C}_x(\mc{E}_+,\mc{E}_-,p)$ are non-monotone in $p$.  
%There is an obvious symmetry in this model. If we write $x=(x^{\sss[1]},\dots,x^{\sss[d]})$ then reflecting the model with parameter $1-p$ about the hyperplane $\{x\in \R^d:\sum_{i=1}^d x^{\sss[i]}=0\}$ gives a realisation of the model with parameter $p$.  For this reason one need only study this model for $p\ge 1/2$.  
\end{example}
When $d=2$, this model was studied in \cite{DRE} and \cite{DRE2}. The left side of Figure \ref{fig:2d_e1e2} shows an example.
We henceforth assume the following, which clearly holds for the orthant model. 
\begin{condition}
\label{cond1}
$d\ge 2$, $e_1\in \uE$, $\oE\subset \mc{E}_+$, and $\uF\supset \mc{E}\setminus \uE$.
%$e_1 \in E_1\subset \mc{E}_+$, and $E_2 \supset \mc{E}\setminus E_1$.
\end{condition}
Remark \ref{rem:subsets} then implies (assuming Condition \ref{cond1}), that 
\begin{equation}
 \mc{C}_x(\uE,\mc{E}\setminus \uE, p)\subset \mc{C}_x(\bs{E},\bs{F},\bs{r},\bs{q},p)\subset  \mc{C}_x(\mc{E}_+,\mc{E}, p).\label{2-valued-bound}
\end{equation}

Observe that under Condition \ref{cond1}, from any site at least one of the arrows in direction $e_1$ or  $-e_2$ is available (the former is available if the local environment is $E_i$ for some $i$, while the latter is available otherwise), so every $\mc{C}_x$ contains an infinite self-avoiding path. 

\begin{example}
\label{exa:maximal}
The case $k=\ell=1$, $E_1=\mc{E}_+$ and $F_1=\mc{E}$ will be referred to as the \emph{half-orthant model}. It is the ``maximal'' model satisfying Condition \ref{cond1}. Since $E_1\subset F_1$,  $\mc{C}_o(p)$ is monotone decreasing in $p$ in this case.  Obviously $\mc{C}_o(0)=\Z^d$ and $\mc{C}_o(1)=(\Z_+)^d$.  It will turn out that there is a non-trivial phase transition for having $\mc{C}_o(p)=\Z^d$.
\end{example}
See the right side of Figure \ref{fig:2d_e1e2} for an illustration of this model, when $d=2$. Likewise we may compare the following two models, with Figure \ref{fig:2d_e1} showing part of a realisation of the environment and the cluster $\mc{C}_o(.5)$ for Examples \ref{exa:e_1a} and \ref{exa:e_1b} in two dimensions.

\begin{example}
\label{exa:e_1a}
If  $k=\ell=1$, $E_1=\{e_1\}$ and $F_1=\mc{E}\setminus \{e_1\}$ then clearly $\mc{C}_o(0)=\Z_-\times \Z^{d-1}$ and $\mc{C}_o(1)=\Z_+ \times \{0\}^{d-1}$ so the sets $\mc{C}_x(p)$ are non-monotone in $p$.  
\end{example}

\begin{example}
\label{exa:e_1b}
Take  $k=\ell=1$, $E_1=\{e_1\}$ and $F_1=\mc{E}$.  The sets $\mc{C}_x(p)$ are  monotone in $p$.  
\end{example}

\begin{figure}
\hspace{-1cm}
\includegraphics[scale=.4]{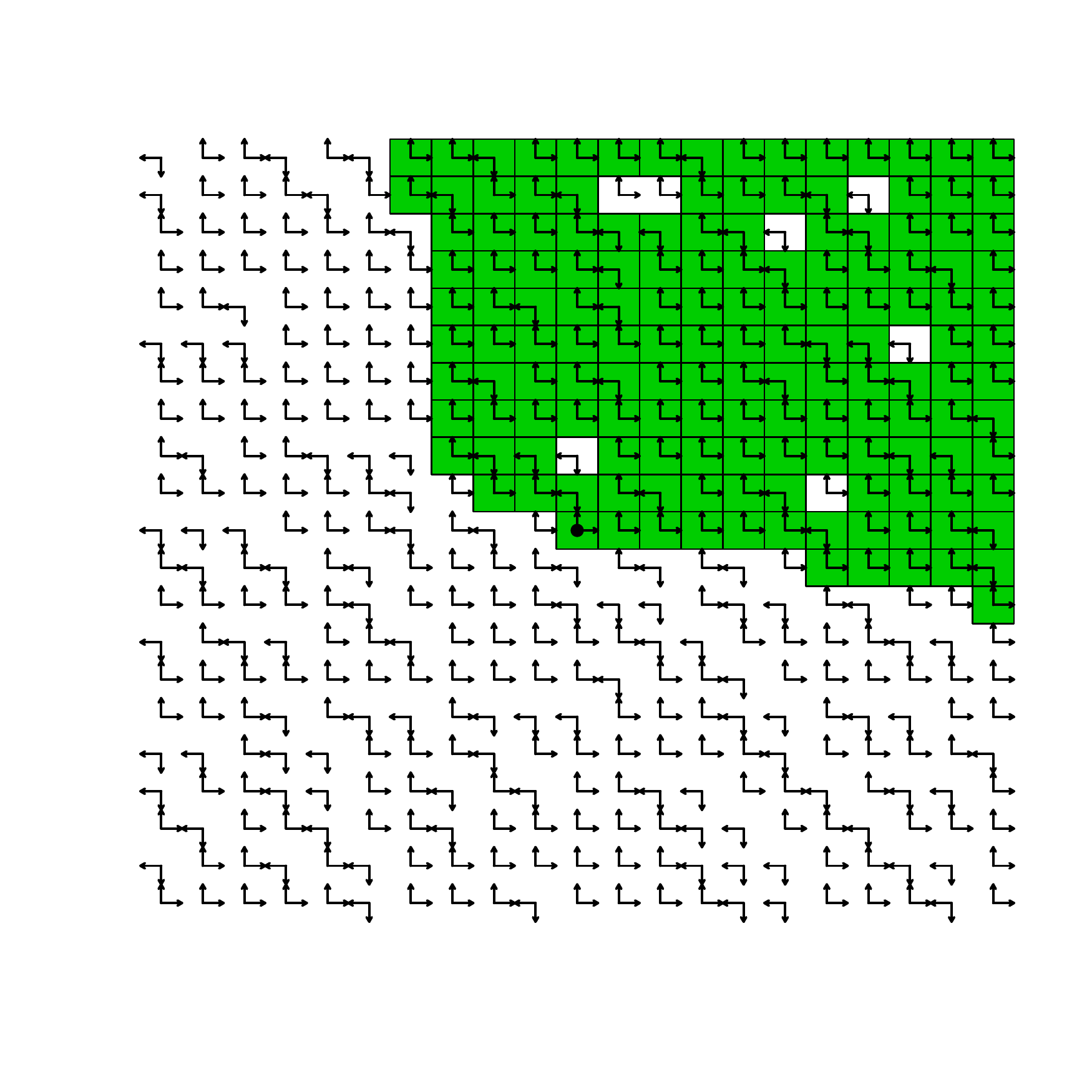}
\includegraphics[scale=.4]{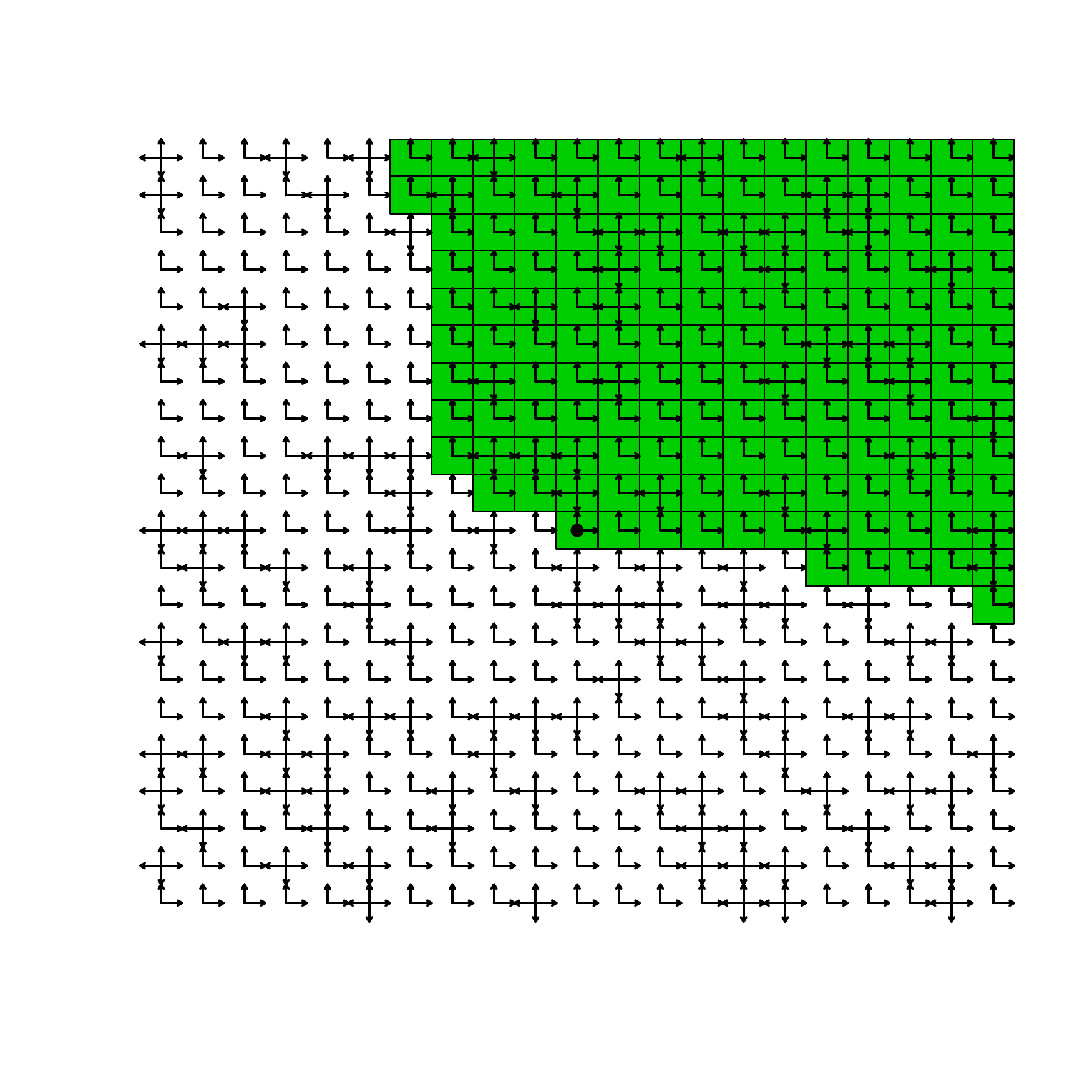}
\vspace{-1.5cm}
\caption{Realisations of finite parts of the set $\mc{C}_o$ for the orthant and half-orthant models (Examples \ref{exa:orthant} and \ref{exa:maximal}) with $p=0.7$ and $d=2$. So $k=\ell=1$ and on the left, $E_1=\{e_1,e_2\}$ and $F_1=\{-e_1,-e_2\}$, while $E_1=\{e_1,e_2\}$ and $F_1=\mc{E}(2)$ on the right. They are generated from the same $U$'s.  Note that the boundaries of the two shaded clusters are the same (see Theorem \ref{thm:main1}).}
\label{fig:2d_e1e2}
\end{figure}

\begin{figure}
\hspace{-1cm}
\includegraphics[scale=.4]{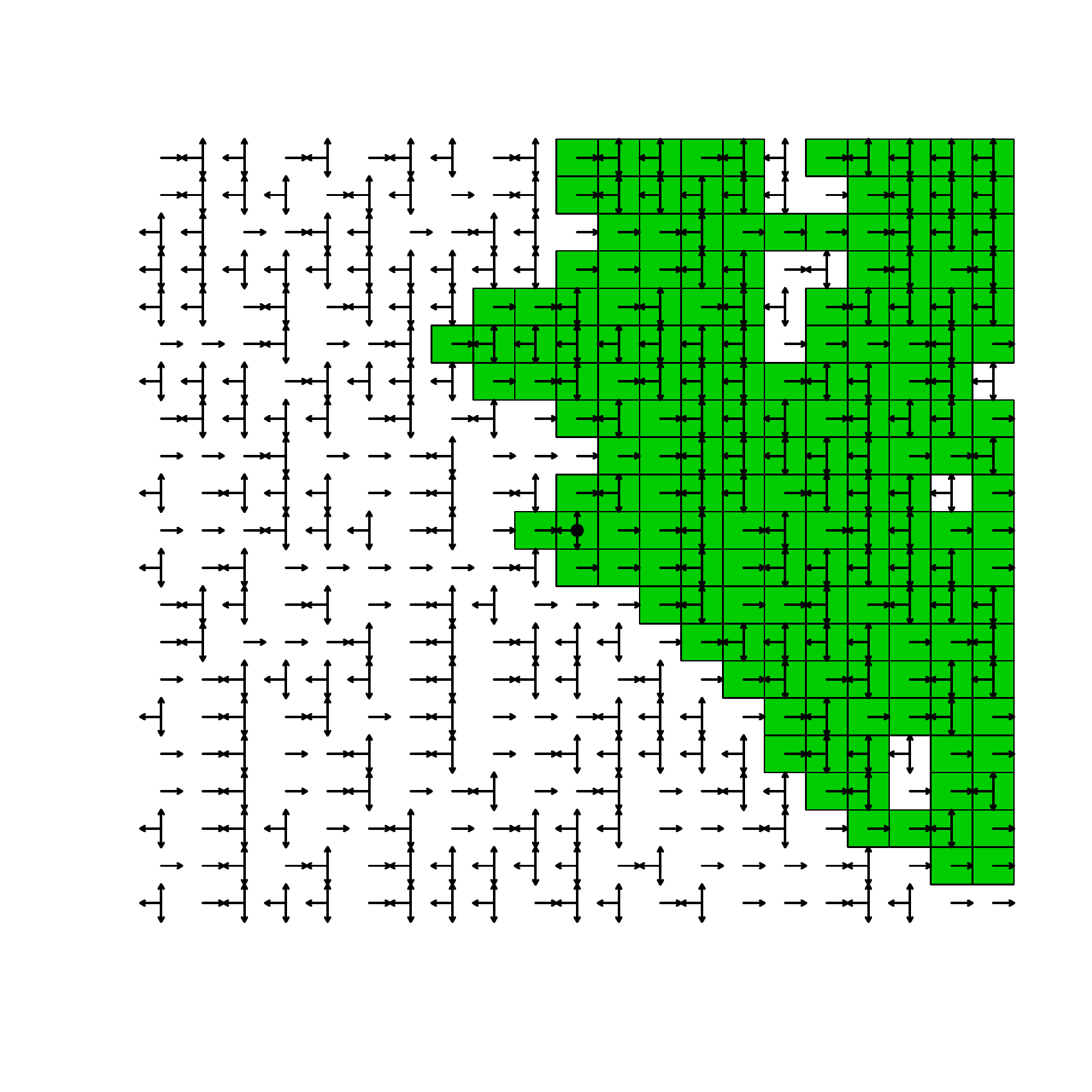}
\includegraphics[scale=.4]{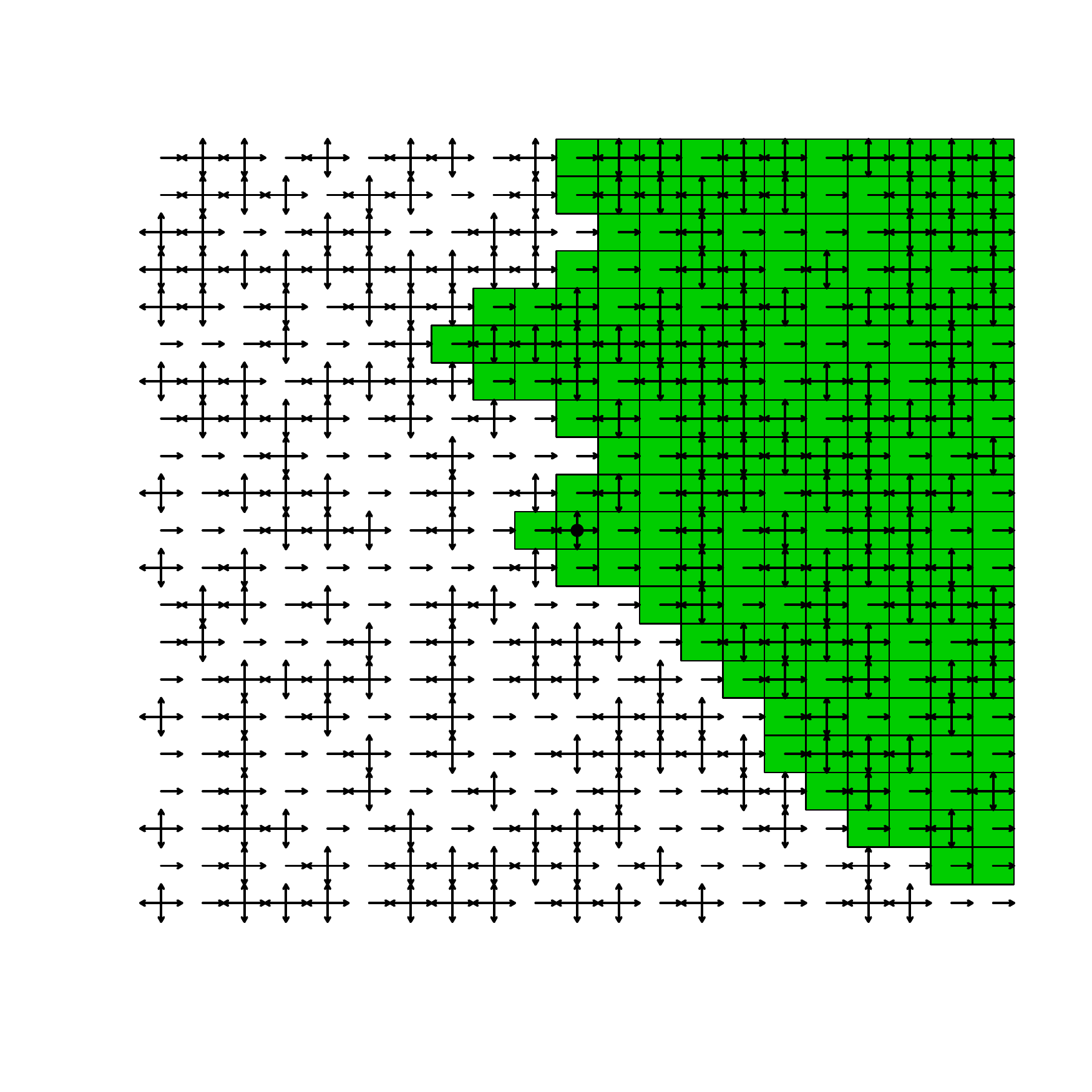}
\vspace{-1.5cm}
\caption{Realisations of finite parts of the set $\mc{C}_o(\frac{1}{2})$ for two models with $k=\ell=1$ and $d=2$. On the left we have $E_1=\{e_1\}$ and $F_1=\{-e_1,e_2,-e_2\}$ and on the right $E_1=\{e_1\}$ and $F_1=\mc{E}(2)$. They are generated from the same environment.  Note that the boundaries of the two shaded clusters are the same (see Theorem \ref{thm:main1}).}
\label{fig:2d_e1}
\end{figure}

%In the orthant model $\mc{C}_x$ is trivially infinite since e.g.~one can follow the North-West (NW) path $\NWP{x}=(x^{\sss(n)})_{n \in \Z_+}$ from $x$:  with $x^{\sss(0)}=x$, take the arrow from $x^{\sss(i)}$ to $x^{\sss(i+1)}=x^{\sss(i)}+e_2$ whenever $\omega_{x^{\sss(i)}}=+$ and otherwise (i.e.~when $\omega_{x^{\sss(i)}}=-$) take the arrow from $x^{\sss(i)}$ to $x^{\sss(i+1)}=x^{\sss(i)}-e_1$.  Then $\NWP{x}$ is an infinite self-avoiding path contained in $\mc{C}_x$.  

%For some of our results we will assume one of the following Conditions which are stronger than Condition \ref{cond1}.
%\begin{condition}
%\label{cond2}
%$e_1 \in E_1\subset \mc{E}_+$, and $E_2=\mc{E}$.
%\end{condition}

For fixed $\bs{F}$, let $\bs{F}^*$ denote the corresponding object with $F_i$ replaced with $\mc{E}$ for each $i$.  Note that we obtain the same model if we take $\ell=1$ and $F_1=\mc{E}$, so we will write
$\mc{C}_x(\bs{E},\bs{r},p)$ for $\mc{C}_x(\bs{E},\bs{F}^*,\bs{r},\bs{q},p)$.
Then, by Remark \ref{rem:subsets},
\begin{equation}
\mc{C}_x(\bs{E},\bs{F},\bs{r},\bs{q},p)\subset \mc{C}_x(\bs{E},\bs{r},p).\label{banana1}
\end{equation}

For $x\in \Z^{d}$ let $L_x:=\inf\{k\in \Z: x+ke_1\in \mc{C}_o\}$.    It is immediate from \eqref{banana1} that (writing $L_x(\bs{E},\bs{r},p)$ for $L_x(\bs{E},\bs{F}^*,\bs{r},\bs{q},p)$)
\begin{equation}
L_x(\bs{E},\bs{F},\bs{r},\bs{q},p)\ge L_x(\bs{E},\bs{r},p).\label{banana2}
\end{equation}
\begin{remark}
\label{rem:L}
If $x=y+ke_1$ for some $k\in \Z$ then $x+je_1\in \mc{C}_o$ if and only if $y+(k+j)e_1\in \mc{C}_o$, so $L_y=L_x+k$.  It follows that for each $y\in \Z^d$, $L_{y+L_ye_1}=0$.
\end{remark}

For $z\in \Z^d$ we define $z_{\{1+\}}=\{z+ke_1:k\in \Z_+\}$, and for $A\subset \Z^d$
\begin{equation}
A_{\{1+\}}=\bigcup_{z \in A}z_{\{1+\}}.\label{A1+}
\end{equation}
Our first main result is the following.  See Figures  \ref{fig:2d_e1e2} and  \ref{fig:2d_e1} for 2-valued illustrations  when $d=2$. See Figure \ref{fig:3d_e_1} for a simulation of a 3-dimensional model. 
\begin{theorem}
\label{thm:main1}
Assume Condition \ref{cond1}.  Then for each  $x\in \Z^{d}$, and $p\in (0,1]$,
\[L_x(\bs{E},\bs{F},\bs{r},\bs{q},p)=L_x(\bs{E},\bs{r},p)\in [-\infty,\infty)\text{ a.s.}\]
and 
\[\mc{C}_o(\bs{E},\bs{F},\bs{r},\bs{q},p)_{\{1+\}}=\mc{C}_o(\bs{E},\bs{r},p)\text{ a.s.}\]
\end{theorem}

%\footnote{\color{red} It seems that the above Theorem can be upgraded to something like:  
%\begin{theorem*}
%\label{thm:main1'}
%Suppose that $e_1\in E_1\subset \mc{E}_+$ and $-e_1\in E_2$, and  $p\in (0,1]$.  Then $(L_x)_{x \in \Z^d}$ is identical for the two models $(E_1,E_2)$ and $(E_1,E_1\cup E_2)$.
%\end{theorem*}
%}
% cor:Co+
Note that it is not true in general that $\mc{C}_o(\bs{E},\bs{F},\bs{r},\bs{q},p)=\mc{C}_o(\bs{E},\bs{r},p)$.  However, roughly speaking Theorem \ref{thm:main1} says that if you only care about the \emph{outer boundary} of $\mc{C}_o$ then under Condition \ref{cond1} you may as well set $\ell=1$ and $F_1=\mc{E}$. Another way of viewing this result is that $\mc{C}_o(\bs{E},\bs{r},p)$ is $\mc{C}_o(\bs{E},\bs{F},\bs{r},\bs{q},p)$ with its holes filled in.

The above results reveal that under Condition 1, a special role is played by the case $\ell=1$, $F_1=\mc{E}(d)$.  For this reason we will state some results in this special case, i.e. assuming the following condition
\begin{condition}
\label{cond2}
$d\ge 2$, $e_1\in \uE$, $\oE\subset \mc{E}_+$, 
$\ell=1$, and $F_1=\mc{E}(d)$.
\end{condition}

%The above results imply that, if we are interested in proving results about $\mc{C}_o(p,E_1,E_2)_{\{1+\}}$, we may as well study the model with $E_2=\mc{E}$.
We now state our second main result which reveals a non-trivial phase transition for the occurrence of the event $\{\mc{C}_o=\Z^d\}$.
\begin{theorem}
\label{thm:main2}
Assume Condition \ref{cond2}.  Then there exists $p_c(\bs{E},d,\bs{r})\in (0,1)$ such that:
\begin{itemize}
\item[] if $p<p_c$ then $\mc{C}_o(\bs{E},\bs{r},p)=\Z^d$ almost surely, and 
\item[] if $p>p_c$ then $L_x(\bs{E},\bs{r},p)$ is finite for every $x\in \Z^{d}$ almost surely (so $\mc{C}_o(\bs{E},\bs{r},p)\ne \Z^d$).
\end{itemize}
\end{theorem}
We conjecture that $\mc{C}_o=\mathbb{Z}^d$ in the case $p=p_c$ as well.

When $d=2$, Theorem \ref{thm:main2} follows from \cite[Propositions 2.3 and 2.4]{DRE2}.  Those results also imply a version ($d=2$ only) of the Theorem under Condition \ref{cond1}, where the conclusion $\mc{C}_o=\Z^2$ (when $p<p_c$) is replaced with 
$\bar{\mc{C}}_o=\Z^2$ (when $p<p_c$),  where $\bar{\mc{C}}_o$ is $\mc{C}_o$ with its finite ``holes'' filled in (and note that all of the holes are finite in 2 dimensions).  In general dimensions we do not know whether all holes in $\mc{C}_o$ are finite.  Theorems \ref{thm:main1} and \ref{thm:main2} seem to be the most natural way of describing the phase transition in general dimensions.

Theorems \ref{thm:main1} and \ref{thm:main2} make use of a dual percolation model. When $d>2$ this is a type of surface percolation. See \cite{GH10, GH12, GHK14} for recent work on other higher dimensional percolation structures.

\begin{figure}
	\begin{tabular}{lcr}
\hspace{-2cm}\includegraphics[scale=1]{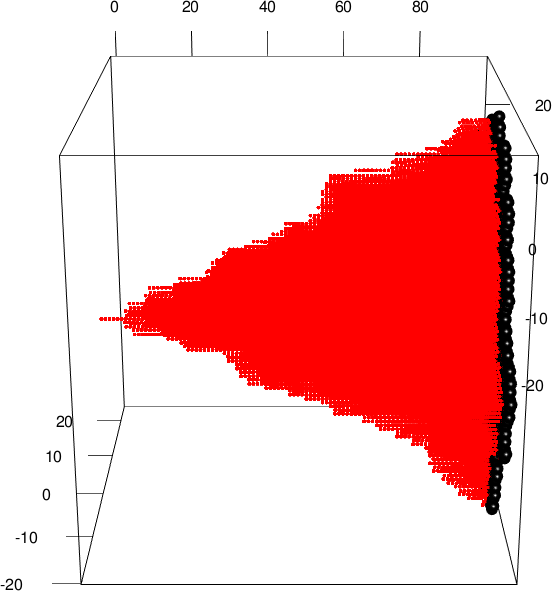}&
\includegraphics[scale=1]{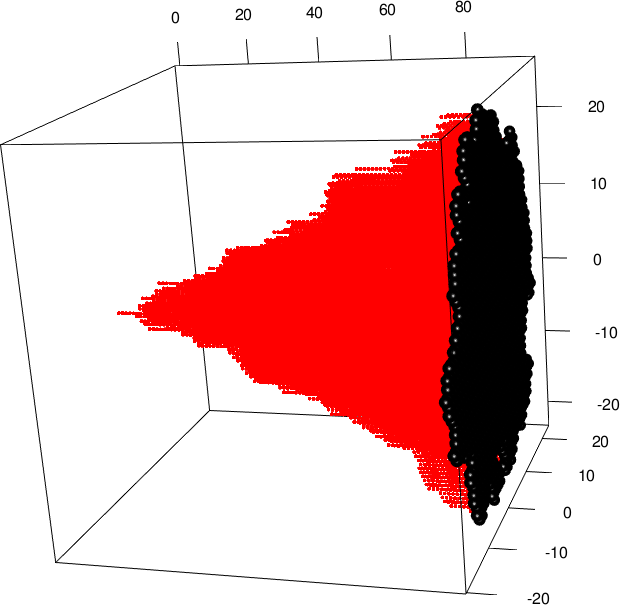}&
\includegraphics[scale=1]{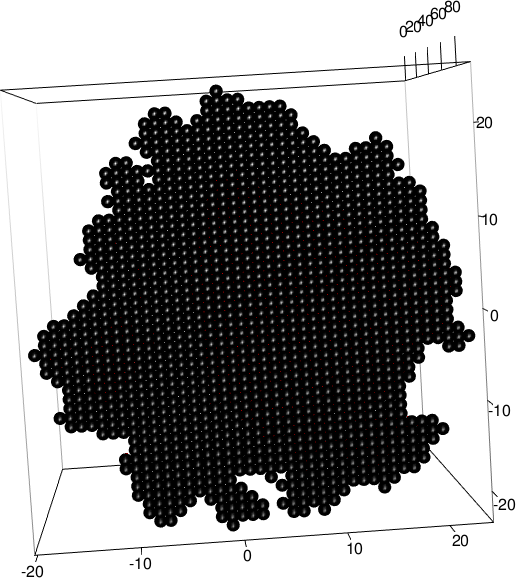}
\end{tabular}
\caption{A simulation of part of the cluster $\mc{C}_o(.9)$ for the model with $d=3$, $k=\ell=1$ and $E_1=\{e_1\}$ and $F_1=\mc{E}(3)$, viewed from 3 different angles.  The black/dark vertices are a cross-section where the first coordinate is equal to $100$.}
\label{fig:3d_e_1}
\end{figure}
%It is intuitively obvious that for fixed $E_1\ni e_1$, $p_c(E_1,d)$ is increasing in $d\ge 2$.  In a separate paper \cite{HS_enhance}, we will prove that $p_c(\mc{E}_+(d),d)$ is strictly increasing.

It is natural to ask about asymptotic properties of the boundary of $\mc{C}_o$ when $p>p_c$.  To this end, let $\Pp$ denote the law of the model with fixed $(\bs{E},\bs{F},\bs{r},\bs{q},p)$ and let $Z$ denote the discrete hyperplane $\{y\in\Z^d: y\cdot e_1=0\}$. 
\begin{open}
\label{open:Wshape}
Fix $d\ge 2$ and assume Condition \ref{cond1}.   Prove that if $p>p_c$ then for each $v\in Z$ there exists a deterministic $\zeta(v)\in \R$ depending on $\bs{E},d,\bs{r},\bs{q},p$ (but not $\bs{F}$) such that 
\begin{equation}
n^{-1}L_{nv}\to \zeta(v), \quad \Pp-\text{almost surely as }n \to \infty.
\end{equation}
%.  If $L_x$ is finite for every $x\in \Z^{d-1}$ then 
\end{open}
In \cite{Shape} a version of this result is proved for Examples \ref{exa:orthant} and \ref{exa:maximal} in general dimensions, though with the assumption that $p$ is sufficiently large.  
%By Theorem \ref{thm:main1}, the same conclusion holds for the orthant model of Example \ref{exa:orthant}.

All of the above results concern the forward cluster $\mc{C}_o$.  A crucial difference between forward and backward clusters is that Condition \ref{cond1} does not ensure that $\mc{B}_o$ is infinite.  In the case of Example \ref{exa:orthant}, if $e_i\in \Omega_+$ and $-e_i\in \Omega_-$ for each $i\in [d]$ (this has positive probability for any $p \in (0,1)$) then there are no arrows pointing to the origin, so $\mc{B}_o=\{o\}$.  Under Condition \ref{cond2} however, $\mc{B}_o$ will be infinite, since it contains $-\mathbb{Z}_+e_1$. 

Another key difference between forward and backward clusters when $d=2$ is that under Condition \ref{cond1} $\mc{B}_o$ is ``simply connected'' as a subset of $\Z^2$, while $\mc{C}_o$ can have holes.  The former does not hold for $d>2$ (see Example \ref{exa:funnyBo} in Section \ref{sec:backwardclusters}).   It seems that for $d>2$ there is no simple geometric description of the possible boundaries for finite $\mc{B}_o$'s.  Infinite  $\mc{B}_o$ clusters appear to be more regular. It would be interesting to characterize infinite clusters that can arise as  $\mc{B}_o$.

For $x\in \Z^{d}$, let $R_x=R_x(\bs{E},\bs{F},\bs{r},\bs{q},p)=\sup\{k \in \Z:x+ke_1\in \mc{B}_o\}$.  The following result shows that under Condition \ref{cond2}, the backward and forward clusters have a phase transition at the same point $p_c$.
\begin{theorem}
\label{thm:B}
Assume Condition \ref{cond2} and let $p_c$ be as in Theorem \ref{thm:main2}.  Then
\begin{itemize}
\item[] if $p<p_c$ then $\mc{B}_o=\Z^d$ almost surely, and 
\item[] if $p>p_c$ then $R_x$ is finite for every $x\in \Z^{d}$.
\end{itemize}
\end{theorem}
We conjecture that $\mc{B}_o=\mathbb{Z}^d$ when $p=p_c$ as well.  The following is an immediate consequence of Theorems \ref{thm:main2} and \ref{thm:B}.
\begin{corollary}
	\label{cor:M}
	Assume Condition \ref{cond2} and let $p_c$ be as in Theorem \ref{thm:main2}.  Then
	\begin{itemize}
		\item[] if $p<p_c$ then $\mc{M}_o=\Z^d$ almost surely, and 
		\item[] if $p>p_c$ then $\mc{M}_o\ne \Z^d$, almost surely.
	\end{itemize}
\end{corollary}
\begin{open}
	Assume Condition \ref{cond2} and let $p_c$ be as in Theorem \ref{thm:main2}.  Show that when $p>p_c$, $\mc{M}_o$ is almost surely finite.
\end{open}

Although $\mc{B}_o$ need not be infinite under Condition \ref{cond1}, this condition is sufficient to ensure that $\mc{B}_o$ is infinite with positive probability (for $p\in (0,1)$), and therefore there exist infinite $\mc{B}_x$ clusters almost surely.  To see this, note that the model contains the $2$-dimensional model with $\mu(\{e_1\})=p$ and $\mu(\{-e_1,-e_2\})=1-p$, for which \cite[Proposition 3.4]{DRE} tells us that $\mc{B}_o$ is infinite with positive probability.  Then the natural analogue of Theorem \ref{thm:main1} is the following.
\begin{open}
\label{open:R*}
Assuming Condition \ref{cond1} show that 
$R_x(\bs{E},\bs{F},\bs{r},\bs{q},p)=R_x(\bs{E},\bs{r},p)$ 
almost surely on the event $\{\mc{B}_o(\bs{E},\bs{F},\bs{r},\bs{q},p) \text{ is infinite}\}$. 
\end{open}
We will prove a partial result in this direction.  
%\M{TIDY UP THE FOLLOWING}
%If $\mc{B}_o=\mc{B}_o(\bs{E},\bs{F},\bs{r},\bs{q},p)$ corresponds to a model satisfying Condition \ref{cond1}, write $\mc{B}_o^*=\mc{B}_o(\bs{E},\bs{r},p)$ for the corresponding model satisfying Condition \ref{cond2}. 
%When $d=2$, the following hold for the orthant model's $\mc{B}_o$ and the half-orthant model's $\mc{B}_o^*$, by arguments of \cite{DRE}: If $1-p_c\le p\le p_c$ then either $\mathcal{B}_o$ is finite or it $=\mathbb{Z}^2=\mc{B}_o^*$. If $p_c<p<1$ then either $\mathcal{B}_o$ is finite or it or it is the region below a decreasing function (and then $\mc{B}_o=\mc{B}_o^*$). If $0<p<p_c$ then either $\mathcal{B}_o$ is finite or it or it is the region above a decreasing function (and then $\mc{B}_o^*=\mathbb{Z}^2$). In all cases, either $\mc{B}_o$ is finite, or all $R_x=R_x^*$. Open Problem \ref{open:R*} postulates that the latter is true in full generality. 
%At present we can only extend the above classification with a weakened version of finiteness. 
For $i\neq 1$ and $z\in\mathbb{Z}^d$, define a family of planes $Z^i(z)=\{x\in\mathbb{Z}^d: x^{\sss[k]}=z^{\sss[k]}\text{ for }k\neq 1,i\}$ and the corresponding $2$-dimensional slices of $\mathcal{B}_o$ as 
$\mathcal{B}_o^{i}(z)=Z^i(z)\cap\mathcal{B}_o$.
Call $\mathcal{B}_o$ {\it semi-finite} if each connected component of each $\mathcal{B}_o^{i}(z)$ is finite. 

If $\mc{B}_o=\mc{B}_o(\bs{E},\bs{F},\bs{r},\bs{q},p)$ corresponds to a model satisfying Condition \ref{cond1}, write $\mc{B}_o^*=\mc{B}_o(\bs{E},\bs{r},p)$ for the corresponding model satisfying Condition \ref{cond2}. Note that $\mc{B}_o^*$ is infinite, since $ke_1\in\mc{B}_o^*$ for every $k\le 0$.
\begin{proposition}
	\label{formofBo} Assume Condition \ref{cond1}. Then either
	\begin{itemize}
		\item[(i)] $\mc{B}_o$ is semi-finite, or 
		\item[(ii)] $\mc{B}_o=\mc{B}_o^*$, or
		\item[(iii)] $\mc{B}_o=(\mc{B}_o)_{\{1+\}}\neq \Z^d$ but $\mc{B}_o^*=\mathbb{Z}^d$.
	\end{itemize}
\end{proposition}
 For comparison, in the case of the orthant model with $d=2$, the corresponding alternatives are respectively that $\mc{B}_o$ is finite; that $\mc{B}_o$ is either $\Z^2$ or the region below a decreasing function; or that $\mc{B}_o$ is the region above a decreasing function. See Proposition 3.8 of \cite{DRE} for a more precise statement. 

It is trivial that $\mc{B}_o$ is connected as a subset of the graph $\Z^d$.  For the complementary cluster, we will show the following.
\begin{proposition}
\label{prop:connectedness} Assume Condition \ref{cond1}. Then either $\mc{B}_o^c$ is empty or $\mc{B}_o^c$ is infinite and connected as a subset of the graph $\Z^d$.
\end{proposition}
Section \ref{sec:backwardclusters} contains further discussion of related questions. 
The remainder of this paper is organised as follows.  In Section \ref{sec:forward clusters} we prove Theorems \ref{thm:main1} and \ref{thm:main2}.  In Section \ref{sec:backwardclusters} we prove Theorem \ref{thm:B} and Propositions \ref{formofBo} and \ref{prop:connectedness}.

\section{Forward Clusters: Proof of Theorems \ref{thm:main1} and \ref{thm:main2}}
\label{sec:forward clusters}
Throughout this section $d\ge 2$, $\bs{E}$, $\bs{F}$, $\bs{r}$, and $\bs{q}$ as in Condition \ref{cond1} are fixed.  
%Some of our results below will assume the following
%\begin{condition}
%\label{cond:1'}
%$d\ge 2$, $e_1 \in E_1\subset \mc{E}_+$, and $E_2 \supset \mc{E}\setminus E_1$.
%\end{condition}

\begin{lemma}
\label{lem:Wfinite1}
Assume Condition \ref{cond1} and let $p\in (0,1)$.  Then 
%(for the model $(E_1, \mc{E}\setminus E_1)$, 
\begin{align}
L_x<\infty \text{ for every $x\in \Z^{d}$, a.s.}  
\end{align}
\end{lemma}
\begin{proof}
Fix $x\in \Z^{d}$. We will construct a self-avoiding path $P\subset \mc{C}_o$ from $o$ to an $x_2$ of the form $x+ke_1$ for some $k\in\Z$.  It follows that $L_x<\infty$ (possibly $L_x=-\infty$). 

Let $J\subset \{2,\dots,d\}$ denote the set of indices $j\in \{2,\dots, d\}$ for which $x\cdot e_j\ge 0$ and $e_j \in \uE$, and let $J'=\{2,\dots, d\}\setminus J$.  If $J=\emptyset$, take $x_1=o$. If not, suppose that $J=\{j_1,\dots,j_k\}$ for some $k\ge 1$.  Construct a path $P_1\subset \mc{C}_o$ from $o$ as follows: whenever we are at an $F$ site, take the step $-e_1$;  whenever we are at an $E$ site take the step $e_{j_1}$, until the $j_1$-st coordinate matches that of $x$,  then continue with $j_2$ etc. Repeat this until we exhaust the coordinates in $J$. Because all its coordinates are monotonic, $P_1$ is self-avoiding, so the environments we see are independent. Therefore we do eventually exhaust the coordinates in $J$, and arrive at a point $x_1$ whose $J$ coordinates match those of $x$.  

If $J'$ is empty then all coordinates of $x_1$ (except the first) match those of $x$, and we are done (with $x_2=x_1$).  Otherwise from the point $x_1$ construct a self-avoiding path $P_2\subset \mc{C}_o$ as follows:  if at an $E$ site, take the step $e_1$. Otherwise, at an $F$ site, take a step that moves some $J'$ coordinate closer to the corresponding coordinate of $x$.  By definition of $J'$, such a step is possible at every $F$ site. 

All coordinates of $P_2$ are monotonic, so as before, this process eventually terminates at some point point $x_2\in \mc{C}_o$, whose coordinates (other than the first) match those of $x$.  Thus $L_x<\infty$ as claimed.

Note that $P_1$ followed by $P_2$ is indeed self-avoiding, despite the fact that the first coordinate initially decreases and then increases, because the last step of $P_1$ is never in the direction $-e_1$.
\end{proof}
\medskip

%Given $y\in \Z^{d-1}$ and $e\in \mc{E}\setminus \{\pm e_1\}$, define $y_e\in \Z^{d-1}$ by
%\begin{equation}
%\label{ye}
%(0,y_e)=(0,y)+e.
%\end{equation}

\begin{lemma}
\label{lem:Wfinite2}
Assume Condition \ref{cond1} and let $p\in (0,1)$.  Let $x\in \Z^{d}$.
%Then (for the model $(E_1, \mc{E}\setminus E_1)$, 
Then almost surely on $\{L_x=-\infty\}$, we have $L_u=-\infty$ for every $u\in\Z^d$.
\end{lemma}
\begin{proof}
Suppose $L_x=-\infty$. Then $x-ke_1\in \mc{C}_o$ for infinitely many $k>0$.  Suppose $x-ke_1\in\Omega_-\cap\mc{C}_o$. Then $x-(k+1)e_1\in \mc{C}_o$ as well. If it $\in\Omega_-$ then $x-(k+2)e_1\in \mc{C}_o$ as well. We may continue in this way till we find some $x-(k+j)e_1\in\Omega_+\cap\mc{C}_o$ (because $p\in(0,1)$ implies that the probability is zero that $x-(k+j)e_1\in\Omega_-$ for every $j>0$). 
In other words, if $x-ke_1\in\Omega_-\cap\mc{C}_o$ for infinitely many $k>0$ then also $x-ke_1\in\Omega_+\cap\mc{C}_o$ for infinitely many $k>0$.

Similarly one can prove the converse.  Thus, on the event that $L_x=-\infty$, we have that infinitely many of the points $\{x-ke_1:k<0\}$ are in $\Omega_+\cap\mc{C}_o$ and infinitely many are in $\Omega_-\cap\mc{C}_o$.  Using the former, we see that infinitely many of $\{x+e-ke_1:k<0\}$ are in $\mc{C}_o$, whenever $e \in \uE$. Using the latter, we see that infinitely many of $\{x+e-ke_1:k<0\}$ are in $\mc{C}_o$, whenever $e \in \mc{E}\setminus \uE$. Thus $L_{x+e}=-\infty$ for any $e\in\mc{E}$. Using this argument repeatedly proves that $L_u=-\infty$ for every $u\in \Z^{d}$.
\end{proof}
Under Condition \ref{cond1}, this shows that $\{L_o\text{ is finite}\}=\{ L_x\text{ is finite for every }x\in\Z^d\}$ almost surely. But note that the zero-one law for these events won't be established till later in this section. 

Lemma \ref{lem:Wfinite2} can be upgraded slightly (though we will not actually make use of this fact).  For $i \in [d]$ and $x\in \Z^{d}$, let  $L_x^{\sss(i)}=\inf\{k\in\Z:x+ke_i\in \mc{C}_o\}$, and note that $L_x^{\sss(1)}=L_x$ by definition.  Then we have the following.
\begin{lemma}
\label{lem:Wfinite3}
Assume Condition \ref{cond1} and let $p\in (0,1)$. Take $x\in \Z^{d}$ and $i\in[d]$.  
%Then (for the model $(E_1, \mc{E}\setminus E_1)$, 
Then almost surely on $\{L_x^{\sss(i)}=-\infty\}$, we have $L_u^{\sss(i)}=-\infty$ for every $u\in\Z^d$.
\end{lemma}
\begin{proof}
Fix $i\in [d]$, $x \in \Z^{d}$, and $e\in \mc{E}$.  For $n \in \Z_+$, let $M_n=\{x-ke_i:k>n\}$.  Let $\mc{C}_o(m)$ denote the set of points $z\in \mc{C}_o$ for which any shortest path in $\mc{C}_o$ from $o$ to $z$ is of length $\le m$.  We explore the sets $\mc{C}_o(m)$ sequentially, starting with $m=0$, for which $\mc{C}_o(0)=\{o\}$.  Clearly $\mc{C}_o(m+1)$ consists of $\mc{C}_o(m)$ together with the points we can reach in one step from $\mc{C}_o(m)$. So $\mc{C}_o(m+1)$ can be identified just using knowledge of the environments at points in $\mc{C}_o(m)$.  

Given that $N_r$ is finite, we define $N_{r+1}=\inf\{m:\mc{C}_o(m)\cap M_{N_r} \ne \varnothing\}$.  If $N_{r+1}$ is finite then we may find some point $y_r\in M_{N_r}\cap\mc{C}_o(N_{r+1})\setminus \mc{C}_o(N_{r+1}-1)$, whose environment has not been explored prior to the iteration $N_{r+1}$. Therefore that environment is independent of what has come before, and we will have $e\in\mc{G}_{y_r}$ with probability at least  $p\wedge (1-p)>0$. If every $N_r$ is finite this gives infinitely many independent opportunities to have $e\in\mc{G}_{y_r}$. 

It follows that almost surely, either: 
\begin{itemize}
\item $N_r$ is infinite for some $r$ (in which case $L^{\sss(i)}_x>-\infty$), or
\item $L^{\sss(i)}_x=-\infty$ and $e\in \mc{G}_{y_r}$ for infinitely many points $y_r\in M_0\cap \mc{C}_o$.
\end{itemize}
The latter case this implies that $y_r+e\in \mc{C}_o$ for infinitely many $r$, so $L_{x+e}^{\sss(i)}=-\infty$.  Repeating this argument proves the result.
\end{proof}

A function $w:\Z^d\to \Z$ is called a {\em side function} if for each $y \in \Z^d$ and $k \in \Z$, 
\begin{equation}
y+w(y)e_1=y+ke_1+w(y+ke_1)e_1.\label{funky}
\end{equation}
In other words, $w$ picks out a single point on each line $y+\Z e_1$. 

\begin{definition}[$(\bs{E},+)$-Barrier]
\label{def:barrier}
Let $S\subset \Z^d$ be a set of points such that 
\begin{itemize}
\item[(s1)] there exists a side function $w$ such that $S=\{y+w(y)e_1:y \in \Z^d\}$,
\item[(s2)] $S \subset \Omega_+$,
\item[(s3)] for $e\in \mc{E}\setminus \{\pm e_1\}$, if $w(y+e)> w(y)$ then for each $k \in [w(y), w(y+e))$, we have $y+ke_1\in \Omega_+$ and $e\notin \mc{G}_{y+ke_1}$.
\end{itemize}
For each $y\in \Z^{d}$ and $e\in \mc{E}\setminus (\uE\cup \{-e_1\})$ define a set $S_{y,e}$ as follows.  If $w(y+e)> w(y)$   let $S_{y,e}=\{y+ke_1: k \in [w(y),w(y+e))\}$; otherwise let $S_{y,e}=\varnothing$.
Define $\bar{S}=S\cup \bigcup_{y,e}S_{y,e}$.  We call any set $\bar{S}$ formed in this way an \emph{$(\bs{E},+)$-barrier} (with side function $w$).  %\emph{Note that knowing whether a set $S$ satisfying (s1) is an $(e_1,+)$-barrier also depends on the configuration at some points outside of (but ``close to'' $S$)}. 
\end{definition}
Note that taking $k=w(y)$ in \eqref{funky} reveals that $w(y+w(y)e_1)=0$ for each $y \in \Z^d$.  Therefore (s1) above could be replaced by $S=\{x\in \Z^d:w(x)=0\}$.

\begin{remark}
\label{rem:F}
Note that whether or not a set of points $T\subset \Z^d$ is an $(\bs{E},+)$ barrier can be determined by observing $(\mc{G}_{z})_{z \in T}$.  It also does not depend on $\bs{F}$. Moreover, if $T$ is an $(\bs{E},+)$ barrier, then by definition, $T\subset\Omega_+$.
\end{remark}
%\M{For $x\in \Z^d$ define $x^{\sss[i]}:=x\cdot e_i$.}
%\begin{itemize}
%%%%\item[(b1)] $\omega_{x}=+$ for each $x\in S$,
%\item[(s1)] for every $y \in \Z^{d-1}$ there is a unique $w=w(y)\in \Z$ such that $(w,y) \in S$,
%%%%\item[(b3)] for every $y,y'\in \Z^{d-1}$, if $(0,y')=(0,y)+e$ for some $e\in \mc{E}\setminus E_1$
%if $y'=y+e\in \Z^{d-1}$ for some $e\in \mc{E}\setminus E_1$ 
%%%%%and $w(y)<w(y')$ then $\omega_{(k,y)}=+$ for each $k \in [w(y),w(y')-1]$,  
%\item[(s2)] for every $y,y'\in \Z^{d-1}$, if $(0,y')=(0,y)+e$ for some $e\in E_1\setminus \{e_1\}$ then $w(y')\le w(y)$,
%\end{itemize}
%is called an {\em $E_1$ skeleton}.  If $z\in \Z^d$ we will define $w(z):=w(z^{\sss[2]},\dots, z^{\sss[d]})$ and similarly $L_z=L_{(z^{\sss[2]},\dots, z^{\sss[d]})}$.
%%%%For fixed $z\in \Z^d$, we say that there exists an $E_1$ skeleton for $z$ if there exists an $E_1$ barrier $S$ such that $w(z)\le z^{\sss[1]}$.  
%\begin{definition}[$E_1$ barrier]
%\label{def:E1barrier}
%Given an $E_1$-skeleton $S=((w(y),y):y \in \Z^{d-1})$, define for each $y,y' \in \Z^{d-1}$ a set $S_{y,y'}$ as follows. 
%If  $(0,y')=(0,y)+e$ for some $e\in \mc{E}\setminus E_1$ and $w(y)< w(y')$ let $S_{y,y'}=\{(k,y): k \in [w(y),w(y')-1]\}$; otherwise let $S_{y,y'}=\varnothing$.
%The set $\bar{S}=S\cup \bigcup_{y,y'}S_{y,y'}$ is called an \emph{$E_1$ barrier}.
%\end{definition}
On $\{ L_x\text{ is finite for every }x\in\Z^d\}$, define 
$
S_L=\{x+L_xe_1:x\in \Z^{d}\}. 
$
By Remark \ref{rem:L} we have $S_L=\{x \in \Z^d:L_x=0\}$.
\begin{lemma}
\label{lem:Wbarrier}
Assume Condition \ref{cond1}.  On $\{ L_x\text{ is finite for every }x\in\Z^d\}$, $\bar{S}_L$ is an $(\bs{E},+)$-barrier with side function $w(y)=L_y$ for each $y\in \Z^d$, and $w(o)\le 0$.
% for the model $(E_1,\mc{E}\setminus E_1)$, 
%(s1)--(s3) will hold for $S_W$ \M{with $w(y)=L_y$ for each $y$}. Thus and $\bar{S}_W\subset \Omega_+$.
\end{lemma}
\begin{proof}
Remark \ref{rem:L} shows that $x\mapsto L_x$ is a side function.  Thus (s1) holds.
%That (s1) holds for $S_L$ is trivial by definition of $S_L$ and the $L_x$ (which we have assumed are finite).  
Now let $z\in S_L$. We must have $z\in \Omega_+$, since if not then $z-e_1\in \mc{C}_o$, which would contradict the definition of $L_{z}$.  This verifies (s2).  

Turning to (s3), suppose that $L_{y+e}> L_{y}$ and $e\ne \pm e_1$. Then $e \notin \uE$ (since if it were, then $y+L_ye_1\in \mc{C}_o$ implies that $y+e+L_ye_1\in \mc{C}_o$, and therefore $L_y\le L_{y+e}$).  Suppose that $y+ke_1\in \Omega_-$ for some $k \in [L_y,\dots, L_{y+e})$, and let $\hat{k}$ be the first such $k$.  Then $y+\hat{k}e_1\in \mc{C}_o$, since $y+L_ye_1\in \mc{C}_o$ and $e_1\in \mc{G}_{y+je_1}$ for each $j\in [L_y,\hat{k})$.  Therefore $y+\hat{k}e_1+e\in \mc{C}_o$ (since $e\notin\uE$), so $L_{y+e}\le \hat{k}<L_{y+e}$, which is impossible.  Therefore $y+ke_1\in \Omega_+$ for each $k \in [L_y,\dots, L_{y+e})$.  

This also shows that $y+ke_1\in \mc{C}_o$ for each such $k$, so in fact $e \notin \mc{G}_{y+ke_1}$ for any such $k$ (otherwise $y+ke_1+e\in \mc{C}_o$ and hence $L_{y+e}\le k$, which is impossible).  This verifies (s3), confirming that $\bar{S}_L$ is an $(\bs{E},+)$ barrier.

Finally, we have $w(o)\le 0$ because $o\in\mc{C}_o\Rightarrow L_o\le 0$. 
%The remaining statements are immediate. 
\end{proof}

\begin{lemma}
\label{lem:try_this}
Assume Condition \ref{cond2}. Whenever $S\subset\Z^d$ satisfies (s1)--(s3) of Definition \ref{def:barrier}, with $w(o)\le 0$, it follows that $L_x\ge w(x)$ for every $x\in \Z^{d}$. Moreover $w(x)\le 0$ for every $x\in\mc{C}_o$.
\end{lemma}
\begin{proof}
Suppose it isn't true that all $L_x\ge w(x)$.  Then there exists some $x\in \Z^{d}$ and a $k<w(x)$ such that $z:=x+ke_1\in \mc{C}_o$.  Thus $w(z)=w(x)-k>0$. There exists a self-avoiding path $y_0,y_1,\dots,y_N$ in $\mc{C}_o$ from $y_0=o$ to $y_N=z$.  Let $z_1$ be the first location $y$ along this path at which $w(y)>0$.  Then $z_1\ne o$ since we've assumed that $w(o)\le 0$. Let $z_2$ denote the location immediately preceding $z_1$ along this path. Then $w(z_2)   \le 0$ and $z_1=z_2+e$ for some $e\ne e_1$.  

We cannot have $e=e_1$, since in that case $w(z_1)=w(z_2)-1<0< w(z_1)$, which is impossible. 

We cannot have $e=-e_1$ either.  If it were, then $w(z_2)-1=w(z_1)<0\le w(z_2)$, so in fact $w(z_2)=0$.  Thus $z_2\in S$, so by (s2) we have $z_2\in \Omega_+$. This implies that $-e_1\notin \mc{G}_{z_2}$ which is impossible, given the definition of $z_2$.

Therefore $e\ne \pm e_1$. We know that $w(z_1)> 0\ge w(z_2)$, or in other words, $0\in[w(z_2),w(z_2+e))$. By (s3) it follows that $z_2\in\Omega_+$ and $e\notin\mc{G}_{z_2}$. This is impossible, given the definition of $z_2$, which establishes that all $L_x\ge w(x)$.

The final conclusion now holds, because if $x\in\mc{C}_o$ then $w(x)\le L_x\le 0$.
\end{proof}

In the next argument, for simplicity, we will write $L_x$ and $L^*_x$ respectively for the objects $L_x(\bs{E},\bs{F},\bs{r},\bs{q},p)$ and $L_x(\bs{E},\bs{r},p)$ of Theorem \ref{thm:main1}. Recall that the former corresponds to a model satisfying Condition \ref{cond1}, and the latter to a corresponding model satisfying Condition \ref{cond2}. We will adopt the same shorthand for other quantities obtained from these model so that, for example, Theorem \ref{thm:main1} is the statement that $L_x=L_x^*$ and $(\mc{C}_o)_{\{1+\}}=\mc{C}_o^*$. 

\begin{proof}[\textbf{Proof of Theorem \ref{thm:main1}}]
If $p=1$ the claim is trivial, so assume $p\in(0,1)$. By \eqref{banana2},  $L^*_x\le L_x$ for every $x\in \Z^{d}$.  Thus if $L_x=-\infty$ for every $x$ then there is nothing to prove.

By Lemma \ref{lem:Wfinite2} we may therefore assume that $L_u>-\infty$ for every $u\in \Z^{d}$.  By Lemma \ref{lem:Wbarrier}, $S_L$ satisfies (s1)--(s3) and by definition, $L_o\le 0$.  By Lemma \ref{lem:try_this} we obtain that $L^*_x\ge L_x$ for every $x\in \Z^{d}$.

Now consider the second assertion. In one direction, the fact that $e_1\in \mc{G}^*_x$ for every $x$ implies that $(\mc{C}_o)_{\{1+\}}\subset (\mc{C}_o^*)_{\{1+\}}=\mc{C}_o^*$. In the other direction, let $x\in\mc{C}_o^*$. Then $L_x^*\le 0$ so by the first part of the Theorem, also $L_x\le 0$. This implies that $x\in(\mc{C}_o)_{\{1+\}}$, and we're done.
\end{proof}

Note that probability enters the above arguments only via Lemmas \ref{lem:Wfinite1} and \ref{lem:Wfinite2}. Outside of those results, the proofs are purely graph-theoretic. We cannot entirely eliminate probability however.  For example, setting $p=0$ in Example \ref{exa:orthant}) gives $(\mc{C}_o)_{\{1+\}}=\{(i,j): j\le 0\}$, whereas $\mc{C}_o^*=\Z^2$ in this case.

\begin{proof}[\textbf{Proof of Theorem \ref{thm:main2}}]
Assume Condition \ref{cond2} and that $p\in (0,1)$. 
Consider the following alternatives:
\begin{itemize}
\item[(i)] $L_x=-\infty$ for every $x \in \Z^d$;
\item[(ii)]  $L_x$ is finite for every $x \in \Z^d$.
\end{itemize}
Lemmas \ref{lem:Wfinite1} and \ref{lem:Wfinite2} show that the event that (i) or (ii) holds has probability one. 

Lemma \ref{lem:Wbarrier}  shows that in case (ii) 
there exists an $(\bs{E},+)$ barrier $\bar{S}_L$ with $w(o)\le 0$.  On the other hand, if there is an $(\bs{E},+)$ barrier $\bar{S}$ with $w(o)\le 0$ then Lemma \ref{lem:try_this} shows that $L_x\ge w(x)$ for every $x\in \Z^{d}$.  In other words, (ii) and the existence of an $(\bs{E},+)$ barrier with $w(o)\le 0$ are equivalent.

The event that there exists an $(\bs{E},+)$ barrier (somewhere) is translation invariant, and by ergodicity of the environment it follows that the probability that there exists an $(\bs{E},+)$ barrier is 0 or 1.  If it is 1, then it follows that the probability that there exists an $(\bs{E},+)$ barrier with $w(o)\le n$ increases to 1 as $n \to \infty$.  By translation invariance, this probability does not actually depend on $n$, hence almost surely there is such an $(\bs{E},+)$ barrier with $w(o)\le 0$.  

We have shown that for each $p\in (0,1)$, either (i) holds almost surely or (ii) holds almost surely.
In case (i), since $e_1$ is always $\in \mc{G}_x$ we have $\mc{C}_o(p)=\Z^d$.  In case (ii), clearly $\mc{C}_o(p)\neq\Z^d$.  Since $\mc{C}_o(p)$ is monotone decreasing in $p$ this proves the existence of a $p_c$ below which $\mc{C}_o(p)=\Z^d$, and above which $\mc{C}_o(p)\neq\Z^d$.

It remains only to show that $p_c\in (0,1)$.  The fact that $p_c>0$ follows because this model dominates $d$-dimensional site percolation (corresponding to setting $E_i=\varnothing$) with parameter $1-p$.  For $1-p$ site percolation the connected cluster of the origin contains infinitely many points of the form $-ke_1$ for $k\ge 0$ with positive probability when $1-p$ is larger than the critical probability $p_c^{\text{site}}<1$ of the model.  It is easy to show that $p_c<1$, by counting self-avoiding walks as in e.g.~\cite[proof of Theorem 4.2]{DRE}.
\end{proof}

  \section{Backward clusters. Proof of Theorem \ref{thm:B}}
  \label{sec:backwardclusters}
It is trivial that $\mc{B}_o$ is connected as a subset of the graph $\Z^d$.  In \cite{DRE,DRE2} for the planar case $d=2$ it is proved that under certain general conditions (implied by Condition \ref{cond1}), $\mc{B}_o$ is \emph{simply} connected (as a subset of the graph $\Z^d$).  This need not be the case (assuming only Condition \ref{cond1}) in higher dimensions as per the following example.
\begin{figure}[t]
\centering
\includegraphics[scale=.4]{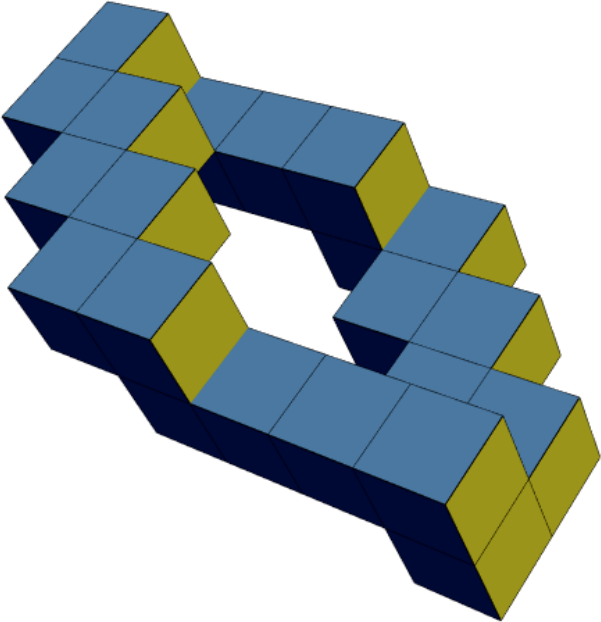}
\caption{A rotation of the irregular (finite) $\mathcal{B}_o$ in Example \ref{exa:funnyBo}.   In this example neither $\mc{B}_o$ nor $\mc{B}_o^c$ are simply connected.}
\label{fig:funnyB}
\end{figure}
\begin{example}
\label{exa:funnyBo}
Consider Example \ref{exa:orthant} in 3 dimensions with $p\in (0,1)$.  Then with positive probability neither $\mathcal{B}_o$ nor $\mathcal{B}_o^c$ is simply connected.  To be precise, in this model, with positive probability $\mc{B}_o$ is precisely the loop $(0,0,0)$, $(-1,0,0)$, $(-2,0,0)$, $(-3,0,0)$, $(-3,0,1)$, $(-3,-1,1)$, $(-2,-1,1)$, $(-2,-2,1)$, $(-1,-2,1)$, $(-1,-3,1)$, $(0,-3,1)$, $(0,-3,0)$, $(1,-3,0)$, $(2,-3,0)$, $(3,-3,0)$, $(3,-3,-1)$, $(3,-2,-1)$, $(2,-2,-1)$, $(2,-1,-1)$, $(1,-1,-1)$, $(1,0,-1)$, $(0,0,-1)$ and back to $(0,0,0)$.   See Figure \ref{fig:funnyB}.  

Here we specify whether various (finitely many) vertices $(x,y,z)\in \Z^3$ are in $\Omega_+$ or $\Omega_-$ (if the environment isn't specified, it isn't relevant for the example, and $^*$ indicates that the vertex is in $\mc{B}_o$).
\begin{itemize}
	\item The following vertices with $z=0$ are in $\Omega_-$:	
\begin{align*}
&(0,0,0)^*, (-1,0,0)^*, (-2,0,0)^*, (3,-3,0)^*,\\
& (0,-1,0), (-1,-1,0), (-2,-1,0),(-3,-1,0),(-4,0,0),(-2,-2,0),\\
&(-1,-2,0),(-1,-3,0),(0,-4,0),(1,-4,0),(2,-4,0),(3,-4,0).
\end{align*}
	\item The following vertices with $z=0$ are in $\Omega_+$:	
\begin{align*}	
&(-3,0,0), 0,-3,0)^*, (1,-3,0)^*, (2,-3,0)^*,\\
&(0,1,0),(1,0,0),(1,-1,0),(2,-1,0),(0,-2,0),(1,-2,0),(2,-2,0),\\
&(3,-2,0),(4,-3,0),(-1,1,0),(-2,1,0),(-3,1,0).
\end{align*}
	\item The following vertices with $z=-1$ are in $\Omega_-$:
	\begin{align*}
	&(1,0,-1)^*,(2,-1,-1)^*,(3,-2,-1)^*\\
	&(-3,0,-1),(-2,0,-1),(-1,0,-1),(0,-1,-1),(1,-2,-1),\\
	&(0,-3,-1),(1,-3,-1),(2,-3,-1),(3,-4,-1).
	\end{align*}	
\item The following vertices with $z=-1$ are in $\Omega_+$:
\begin{align*}
&(0,0,-1)^*,(1,-1,-1)^*,(2,-2,-1)^*,(3,-3,-1)^*\\
&(0,1,-1),(1,1,-1),(2,0,-1),(3,-1,-1),\\
&(4,-2,-1),(4,-3,-1).
\end{align*}	
\item The following vertices with $z=1$ are in $\Omega_-$:
\begin{align*}
&(-3,0,1)^*,(-2,-1,1)^*,(-1,-2,1)^*,(0,-3,1)^*,\\
&(-4,0,1),(-4,-1,1),(-3,-2,1),(-2,-3,1),(-1,-4,1),(0,-4,1),
\end{align*}
\item The following vertices with $z=1$ are in $\Omega_+$:
\begin{align*}
&  (-3,-1,1)^*,(-2,-2,1)^*,(-1,-3,1)^* \\
&(-3,1,1),(-2,0,1),(-1,0,1),(0,0,1),(-1,-1,1),\\
&(0,-2,1),(1,-3,1),(2,-3,1), (3,-3,1).
\end{align*}
\item For any vertex $(x,y,1)^*$ appearing above set $(x,y,2)\in \Omega_+$, and for any vertex $(x,y,-1)^*$ appearing above, set $(x,y,-2)\in \Omega_-$.
\end{itemize}
%What happens in this example is that $\mathcal{B}_o$ consists of a loop $(0,0,0)$, $(-1,0,0)$, $(-2,0,0)$, $(-3,0,0)$, $(-3,0,1)$, $(-3,-1,1)$, $(-2,-1,1)$, $(-2,-2,1)$, $(-1,-2,1)$, $(-1,-3,1)$, $(0,-3,1)$, $(0,-3,0)$, $(1,-3,0)$, $(2,-3,0)$, $(3,-3,0)$, $(3,-3,-1)$, $(3,-2,-1)$, $(2,-2,-1)$, $(2,-1,-1)$, $(1,-1,-1)$, $(1,0,-1)$, $(0,0,-1)$ and back to $(0,0,0)$. Whereas every neighbouring environment points away from this path. In particular, $\mathcal{M}_o=\mathcal{B}_o$ too. Observe that the intersection of $\mathcal{B}_o$ with a horizontal plane $z=0$ gives a disconnected set, which can't happen when $d=2$. In 2 dimensions the intersection of $\mathcal{B}_o$ with horizontal or vertical lines was always a line segment (possibly infinite). That isn't true here either. 
\end{example}

Recall \eqref{A1+} and define $z_{\{1-\}}=\{z-ke_1:k\in \Z_+\}$, and for $A\subset \Z^d$
\begin{equation}
A_{\{1-\}}=\bigcup_{z \in A}z_{\{1-\}}.\label{A1-}
\end{equation}

\begin{proof}[\textbf{Proof of Theorem \ref{thm:B}}]
By Condition \ref{cond2}, $e_1\in \mc{G}_z$ for every $z$.  Thus  $z_{\{1-\}}\subset \mc{B}_z$ for every $z\in \Z^d$. Hence if $z\in \mc{B}_o$ then $z_{\{1-\}}\subset \mc{B}_o$.  

Suppose that $z\in \mc{B}_o$ and let $e \ne \pm e_1$.  Since $p\in (0,1)$ we have that infinitely many points in $(z+e)_{\{1-\}}$ contain $-e$ almost surely, and all of these points are therefore in $\mc{B}_o$ as well.  This proves that $R_x>-\infty$ a.s. for every $x\in \Z^{d}$.  

If $R_{x}=\infty$ for some $x\in \Z^{d}$ then in fact $x+\Z e_1\subset \mc{B}_o$ and therefore for each $e\ne \pm e_1$ almost surely infinitely many points of the form $x+e+ke_1$ with $k\ge 0$ are in $\mc{B}_o$ as well. Therefore $R_{x+e}=\infty$ a.s. on $\{R_x=\infty\}$.  This proves that $\{\text{every $R_x$ is finite}\}\cup\{\text{every $R_x=\infty$}\}$ has probability one.

If $p>p_c$ then with probability 1 there exists an $(\bs{E},+)$ barrier with $w(o)> 0$.  Lemma \ref{lem:try_this} shows that if $z\in \Z^{d}$ satisfies $w(z)\le 0$,  then every $x\in\mc{C}_z$ will satisfy $w(x)\le 0$. This implies that $o\notin\mc{C}_z$. Thus no such $z$ can lie in $\mc{B}_o$, and therefore $R_x<w(x)$. So $R_x$ is finite for every $x \in \Z^{d}$.  

On the other hand, suppose $p<p_c$. If $(R_x)_{x\in \Z^{d}}$ are all finite then let $S=\{x+(R_x+1)e_1:x\in \Z^{d}\}$.  We claim that $\bar{S}$ is an $(\bs{E},+)$ barrier, with $w(x)=R_x+1$ and $w(o)>0$.  But that is a contradiction, since no such barrier exists when $p<p_c$. So in fact, all the $R_x$ will be infinite.  Therefore it only remains to prove that $\bar{S}$ is an $(\bs{E},+)$ barrier.

To prove this, note first that for each $y\in x+\Z e_1$ and $k \in \Z$ we have by definition that $R_{y+ke_1}=R_y-k$.  Therefore $w(x):=R_x+1$ is side function, so (s1) is satisfied.   Next, $S\subset \Omega_+$ since for any $z\in S$, $z-e_1\in \mc{B}_o$ but $z\notin \mc{B}_o$ so $-e_1\notin \mc{G}_z$.  So (s2) holds. Finally, suppose $e \ne \pm e_1$ and $w(y+e)>w(y)$. We know that $(y+e)+ke_1\in \mc{B}_o$ for every $k<w(y+e)$, while $y+ke_1\notin\mc{B}_o$ for $k\ge w(y)$.  This implies that $e\notin \mc{G}_{y+ke_1}$ for any $k \in [w(y),w(y+e))$. Therefore also $y+ke_1\in \Omega_+$ for such $k$, which shows (s3).
\end{proof}

Before we prove Proposition \ref{formofBo} we will state and prove several Lemmas that together will imply the proposition.   For $y \in \Z^d$ and $k_1,k_2\in \Z$ with $k_1<k_2$, let 
%$\Z_{k_1,k_2}(y)=\{y+ke_1: k \in [k_1,k_2]\}$.
$y[k_1,k_2]=\{y+ke_1: k \in [k_1,k_2]\}$.
\begin{lemma}
	\label{lem:segment}
	Suppose that for some $y\in \Z^d$, $k_1<k_2\in \Z$ and $e \in \mc{E}\setminus\{\pm e_1\}$ we have 
	%$\Z_{k_1,k_2}(y)\subset \mc{B}_o$ 
	$y[k_1,k_2]\subset \mc{B}_o$ and $\{y+k_1e_1+e,y+k_2e_1+e\}\subset \mc{B}_o$.  Then 
	%$\Z_{k_1,k_2}(y+e)\subset \mc{B}_o$.
	$(y+e)[k_1,k_2]\subset \mc{B}_o$.
\end{lemma}
\begin{proof}
Either $-e \in \underline{E}$ or $-e\in \underline{F}$.  In the first case we have that $(y+e)[k_1,k_2]\cap \Omega_+\subset \mc{B}_o$, and since $-e_1\in \underline{F}$ it then follows that $(y+e)[k_1,k_2]\subset \mc{B}_o$.  In the second case we have that $(y+e)[k_1,k_2]\cap \Omega_-\subset \mc{B}_o$, and since $e_1\in \Omega_+$ it follows that $(y+e)[k_1,k_2]\subset \mc{B}_o$.
\end{proof}

For $y \in \Z^d$, let $y[\Z]=y+\Z e_1$, $y[\Z_+]=y+\Z_+ e_1$, and $y[\Z_-]=y+\Z_- e_1$.
\begin{lemma}
	\label{lem:Zline}
Almost surely, if there exists $y \in \Z^d$ such that $y[\Z]\subset \mc{B}_o$ then $\mc{B}_o=\Z^d$.
\end{lemma}
\begin{proof}
Let $e \in \mc{E}\setminus\{-e_1,e_1\}$.  It suffices to show that $(y+e)[\Z]\subset \mc{B}_o$.  

Either $-e\in \underline{E}$ or $-e\in \underline{F}$.  Since $p\in (0,1)$ we have that infinitely many points $z$ in $(y+e)[\Z_+]$ have $-e\in\mc{G}_z$ and likewise, infinitely many points $z$ in $(y+e)[\Z_-]$ have $-e\in\mc{G}_z$. Each such point is therefore in $\mc{B}_o$.  It follows from Lemma \ref{lem:segment} that $(y+e)[\Z]\subset \mc{B}_o$ as claimed.
\end{proof}

\begin{lemma}
	\label{lem:goat1} 
	%The following holds almost surely.
Suppose that there exists $y \in \Z^d$ such that  $y[\Z_-]\subset \mc{B}_o$, but $y[\Z]\not\subset \mc{B}_o$.  Then almost surely for every $x\in \Z^d$ there exists $K_x\in \Z$ such that $\mc{B}_o\cap x[\Z]=(x+K_x)[\Z_-]$.
%Then for every $y'\in \Z^d$ there exists $k_{y'}\in \Z$ such that $\Z_-(y'+k_{y'}e_1)\subset \mc{B}_o$.  If also $\mc{B}_o \ne \Z^d$ then there exists such a $k_{y'}$ such that $y'+(k_{y'}+1)e_1\notin \mc{B}_o$.
\end{lemma}
\begin{proof}
Let $e \in \mc{E}\setminus \{\pm e_1\}$.  Since $y[\Z_-]\subset \mc{B}_o$ and $-e\in \underline{E}\cup \underline{F}$ we have that infinitely many points in $(y+e)[\Z_-]$ are also in $\mc{B}_o$. It follows from Lemma \ref{lem:segment} that $(y+e+ke_1)[\Z_-]\subset \mc{B}_o$ for some $k\in \Z$.  Repeating this argument as needed proves that for every $x\in \Z^d$ there exists $k_x\in \Z$ such that $(x+k_xe_1)[\Z_-]\subset \mc{B}_o$.  Since $y[\Z]\not\subset \mc{B}_o$,  Lemma \ref{lem:Zline} tells us that there is a largest such $k_x$, which we denote by $K_x$.  

If there was any $k>K_x$ such that $x+ke_1\in\mc{B}_o$ then Lemma \ref{lem:segment} would imply that $(x+ke_1)[\Z_-]\subset \mc{B}_o$. This would contradict the definition of $K_x$, so in fact $\mc{B}_o\cap x[\Z]=(x+K_x)[\Z_-]$.

%Let $A=\cup_{x \in \Z^d}(x+K_xe_1)[\Z_-]$.  Then $A\subset \mc{B}_o$.   

%Suppose that the final claim of the Lemma is false.  Then there exists some $w\in \mc{B}_o$ such that $w\notin A$.   Since $\mc{B}_o$ is connected there exists a finite self avoiding path (not necessarity consistent with the environment) from $A$ to $w$ consisting of sites in $\mc{B}_o$.  Let $w_0$ be the first site on that path such that $w_0\notin A$.  Then there exists some $e\in \mc{E}\setminus \{\pm e_1\}$ such that $w_e:=w_0-e\in A$.  Since $w_e\in A$ we have that $\Z_-(w_e)\subset \mc{B}$.  Since infinitely many points in $\Z_-(w_e)+e$ contain $-e$, we know that infinitely many points in $\Z_-(w_e)+e$ are in $\mc{B}_o$.  Since $w_0=w_e+e\in \mc{B}_o$ by assumption, Lemma \ref{lem:segment} implies that $\Z_-(w_0)\subset \mc{B}_o$.  This contradicts the definition of $w_0$
\end{proof}
The proof of the following is similar, and is left to the reader.
\begin{lemma}
	\label{lem:goat2}
	Suppose that there exists $y \in \Z^d$ such that  $y[\Z_+]\subset \mc{B}_o$ but $y[\Z]\not \subset \mc{B}_o$.  Then almost surely for every $x\in \Z^d$ there exists $K'_x\in \Z$ such that $\mc{B}_o\cap x[\Z]=(x+K'_x)[\Z_+]$.
\end{lemma}
Obviously, under the assumptions of Lemma \ref{lem:goat1} we have $K_{x+e_1}=K_x-1$, hence $w(x):=K_x+1$ satisfies \eqref{funky}.  Similarly, in Lemma \ref{lem:goat2} we have $K'_{x+e_1}=K'_x-1$.

Recall the notation $Z^i(z)$ and $\mc{B}^i_o(z)$ given prior to the statement of Proposition \ref{formofBo}.  
\begin{lemma}
	\label{lem:intervals}
Let $i\ne 1$ and $z\in \Z^d$.  Then the set $Z^i(z)\setminus \mc{B}_o$ is connected. 
%as a subset of $Z^i(z)$.

Suppose that $B$ is an infinite connected component of $\mc{B}_o^i(z)$.  Then for any $z'\in Z^i(z)$ the set
\begin{equation}
I=I(z'):=z'[\Z]\setminus B,
	\end{equation}
	is  a single interval (which is possibly empty or infinite, but not bi-infinite).
\end{lemma}
\begin{proof}
For $y_1,y_2\in Z^i(z)\setminus \mc{B}_o$, we can follows paths consistent with the environment and consisting of only steps $e_1$ (from $\Omega_+$ sites) and $-e_i$ (from $\Omega_-$ sites) that eventually intersect (as in Proposition 3.8 of \cite{DRE}).  These paths lie entirely in $Z^i(z)\setminus \mc{B}_o$ since only moves $e_1,-e_i$ were used and  $y_1,y_2\in Z^i(z)\setminus \mc{B}_o$.  This proves the first claim.

For the second claim, suppose that $I$ is not an interval.  Then there exist $y_1,y_2\in I$ with $y_1^{\sss[1]}<y_2^{\sss[1]}-1$ and such that $v \in B$ for every $v \in z'[\Z]$ with $y_1^{\sss[1]}<v^{\sss[1]}<y_2^{\sss[1]}$.  Then $y_1,y_2\notin \mc{B}_o$, since they neighbour $B$ but are $\notin B$.  From $y_1$ and $y_2$ we may follow paths consistent with the environment using only $e_1$ and $-e_i$ moves from $\Omega_+$ sites and $\Omega_-$ sites respectively.  These paths eventually meet (again, as in Proposition 3.8 of \cite{DRE}) and are contained in $Z^i(z)\setminus \mc{B}_o$.  	Similarly, if $e_i\in \underline{E}$ then from $y_1$ and $y_2$ we may follow paths consistent with the environment using only $-e_1$ and $e_i$ moves, from $\Omega_-$ sites and $\Omega_+$ sites respectively.  If $e_i\notin \underline{E}$ then we may instead follow paths using only $e_i$ and $e_1$ moves,  from $\Omega_-$ sites and $\Omega_+$ sites respectively.  In either case the two paths intersect and are contained in $\mc{B}_o^c$.

It follows that each $v$ as above is enclosed by a circuit in $Z^i(z)\setminus \mc{B}_o$ and hence $v$ is not in an infinite component of $\mc{B}_o\cap Z^i(z)$, contradicting that $v\in B$.  This shows that $I$ is indeed an interval.

It remains only to prove that $I(z')\ne z'[\Z]$.  If $\mc{B}_o=\Z^d$ then this holds, since $I(z')=\emptyset$. So assume this is not the case.
Since $B$ is non-empty there is some $u\in \Z^d$ such that $u[\Z]\cap B$ is non-empty, and since $I(u)$ is an interval, $u[\Z]\cap B$ must contain a half line. Without loss of generality it is $y[\Z_-]$ for some $y\in u[\Z]$.  Because $\mc{B}_o\neq\Z^d$, Lemma \ref{lem:Zline} implies that $y[\Z]\not\subset B_o$.  So by Lemma \ref{lem:goat1}, for every $x\in Z^i(z)$ there is a $K_x\in \Z$ such that $\mc{B}_o \cap x[\Z]=(x+K_xe_1)[\Z_-]$.  But $\bigcup_{x\in Z^i(z)}(x+K_xe_1)[\Z_-]$ is connected in $Z^i(z)$ and intersects $B$ (just take $x=u$), so in fact it is equal to $B$. 
Thus $B$ intersects $z'[\Z]$, so $I(z')\ne z'[\Z]$.
\end{proof}

\begin{proof}[\textbf{Proof of Proposition \ref{formofBo}}]
If $v[\Z]\subset \mc{B}_o$ for some $v\in \Z^d$ then $\mc{B}_o=\Z^d$ by Lemma \ref{lem:Zline}, and we are in case (ii). So assume this is not the case.

If all components of $\mathcal{B}_o^{i}(z)$ are finite for all $i$ and $z$, then $\mc{B}_o$ is semi-finite, and we are in case (i).

Otherwise, for some $i\ne 1$ and $z\in\mathbb{Z}^d$ the set $\mathcal{B}_o^{i}(z)$ has an infinite component $B$. Without loss of generality we assume $i=2$.
By Lemma \ref{lem:intervals}, for each $z'\in Z^2(z)$ we have that $I(z')=z'[\Z]\setminus B$ is an interval that is not bi-infinite, so there exists $z''\in \Z(z')$ such that either $z''[\Z_-]\subset B$ or  $z''[\Z_+]\subset B$.  

In the second case, since $z''[\Z_+]\subset B\subset \mc{B}_o$ but $z''[\Z]\not \subset \mc{B}_o$, by Lemma \ref{lem:goat2} we have that each $\mc{B}_o\cap x[\Z]$ has the form $(x+K'_x)[\Z_+]$ for some $K'_x\in\Z$.  It follows that $\mc{B}_o=(\mc{B}_o)_{\{1+\}}$ and $\mc{B}_o^*=\Z^d$, so case (iii) of the Proposition holds.

In the first case, Lemma \ref{lem:goat1} implies that for every $x$, $\mc{B}_o\cap x[\Z]$ has the form $(x+K_x)[\Z_-]$ for each some $K_x\in\Z$.  Let $S=\{x+(K_x+1)e_1:x \in \Z^d\}$.  Since $-e_1 \notin \mc{G}_{x+(K_x+1)e_1}$ by definition of $K_x$ we have that (s1) and (s2) of Definition \ref{def:barrier} hold for $S$ (with $w(x):=K_x+1$).  Suppose that $w(y+e)>w(y)$, and let $J=[w(y),w(y+e))$.  Then $e\notin \cup_{j \in J}\mc{G}_{y+je_1}$, so either  $\{y+je_1:j \in J\}\subset \Omega_+$ (if $e\notin\underline{E}$) or $\{y+je_1:j \in J\}\subset \Omega_-$ (if $e\in\underline{E}$).  

The second alternative cannot occur, since if it did then $-e_1\in \mc{G}_{y+je_1}$ for each $j \in J$, so $\{y+je_1:j \in J\}\subset \mc{B}_o$ which contradicts the definition of $w(y)$.  Therefore $\{y+je_1:j \in J\}\subset \Omega_+$, so (s3) of Definition \ref{def:barrier} holds, i.e. $\bar{S}$ is an $(\bs{E},+)$ barrier with side function $w$.   It remains an  $(\bs{E},+)$ barrier with side function $w$ when we change all $\Omega_-$ sites to $\mc{E}$, so this proves that $\mc{B}_o=\mc{B}_o^*$ in this case.

\end{proof}

\begin{proof}[\textbf{Proof of Proposition \ref{prop:connectedness}}]
If $\mc{B}_o^c$ is non-empty then there exists $y\in \mc{B}_o^c$, and since $\mc{C}_y\subset \mc{B}_o^c$ we conclude that $\mc{B}_o^c$  is infinite.

To show connectedness, let $y_1,y_2\in\mc{B}_o^c$. We will construct self-avoiding paths from each, consistent with the environment, that eventually meet. By definition, both paths must lie in $\mc{B}_o^c$, which will establish the result. 

Without loss of generality, $y_1^{\sss[d]}\le y_2^{\sss[d]}$. Build a path from $y_2$ by following $e_1$ at sites in $\Omega_+$ and $-e_d$ at sites in $\Omega_-$ till we reach a point $y'_2\in\mc{B}_o^c$ whose $d$'th coordinate agrees with that of $y_1$. Let $y'_1=y_1$. Repeating the same argument, now starting from $y'_1$ and $y'_2$, we will in turn reach points whose $d$'th and $(d-1)$'st coordinates agree. Continuing in this way, we'll reach points $x_1, x_2\in\mc{B}_o^c$, all of whose coordinates agree, other than the first two. In the notation from before Proposition \ref{formofBo} we'll have that $x_1, x_2$ belong to the plane $Z^2(x_1)$. 

But from $x_1$ and $x_2$ we may now apply the planar construction of Proposition 3.8 of \cite{DRE} (also used above in the proof of Lemma \ref{lem:intervals}) to build paths in $Z^2(x_1)\cap\mc{B}_o^c$ that eventually cross. Thus $y_1$ and $y_2$ both connect to that crossing point.
\end{proof}

\subsection*{Acknowledgements}
The work of MH is supported by Future Fellowship FT160100166 from the Australian Research Council.  The work of TS is supported by NSERC and the Fields Institute.

\bibliographystyle{plain}

\end{document}